\documentclass[10pt]{article}

\usepackage{amsfonts}
\usepackage{amsmath}
\usepackage{amssymb}
\usepackage{amsthm}
\usepackage{bussproofs}
\usepackage{fancyhdr}

\input xy
\xyoption{all}

\newtheorem{definition}{Definition}[section]

\newtheorem{lemma}[definition]{Lemma}
\newtheorem{theorem}[definition]{Theorem}
\newtheorem{remark}[definition]{Remark}
\newtheorem{corollary}[definition]{Corollary}

\newtheorem*{maintheorem}{Main Theorem}

\def\NN{\mathbb{N}}

\def\BB{\mathbb{B}}
\newcommand{\True}{0}
\newcommand{\False}{1}

\newcommand{\PAomega}{{\rm PA}^\omega}

\newcommand{\ACA}{{\sf ACA_0}}

\newcommand{\CA}{{\sf AC}^0}

\newcommand{\pCA}{\Pi^0_1\mbox{-}{\sf AC}}

\newcommand{\WKL}{{\sf WKL}}

\newcommand{\RAM}[2]{{\sf RT}^{#1}_{#2}}

\newcommand{\RAMPAIR}[1]{\RAM{2}{2}({#1})}
\newcommand{\IPP}{{\sf IPHP}}

\newcommand{\BR}{{\sf BR}}

\newcommand{\systemT}{{\rm T}}
\newcommand{\ER}[1]{{\sf E/R}({#1})}

\newcommand{\INF}{{\sf Depth}}

\newcommand{\cZero}{{\bf 0}}

\newcommand{\emptyseq}{\langle \, \rangle}

\newcommand{\eqleft}[1]{\begin{itemize} \item[] $#1$ \end{itemize}}
\newcommand{\pair}[1]{\langle #1 \rangle}

\newcommand{\initSeg}[2]{[#1](#2)}

\newcommand{\ext}[1]{\widehat{#1}}

\newcommand{\game}{\mathcal{G}}

\newcommand{\selEmb}[1]{\overline{#1}}

\newcommand{\nd}[3]{|{#1}|^{#2}_{#3}}

\newcommand{\lth}[2]{{#1}^{\BB^{{#2}}}}
\newcommand{\mt}{T'}
\newcommand{\sbool}{\alpha}
\newcommand{\swit}{\beta}

\newcommand{\is}[1]{[{#1}]}

\newcommand{\mona}{\textup{M}1}
\newcommand{\monb}{\textup{M}2}

\newcommand{\EPS}[3]{{\sf EPS}_{#1}^{#2}({#3})}
\newcommand{\EPSs}{{\sf EPS}}

\fancypagestyle{first}{
\fancyhf{} 
\lhead{Final version to appear in Math. Struct. in Comp. Science} } 

\begin{document}

\title{A Constructive Interpretation of Ramsey's Theorem via the Product of Selection Functions}

\author{Paulo Oliva and Thomas Powell \\ \\ \emph{School of Electronic Engineering and Computer Science,} \\ \emph{Queen Mary University of London}}

\date{}

\maketitle

\begin{abstract}We use G\"{o}del's dialectica interpretation to produce a computational version of the well known proof of Ramsey's theorem by Erd\H{o}s and Rado. Our proof makes use of the product of selection functions, which forms an intuitive alternative to Spector's bar recursion when interpreting proofs in analysis. This case study is another instance of the application of proof theoretic techniques in mathematics.\end{abstract}

\thispagestyle{first}

\section{Introduction}
\label{sec-intro}

In a fundamental paper of the 1950s \cite{Kreisel(51),Kreisel(52)}, Kreisel first suggested utilising proof interpretations to systematically `unwind' non-constructive proofs and discover their constructive content. Kreisel's pioneering work forms the foundation of modern applied proof theory (in the sense of \cite{Kohlenbach(2008)}) which has seen variants of G\"{o}del's functional interpretation used to produce improved results in areas such as numerical analysis and ergodic theory through the extraction of computational content from classical proofs. 

This \textit{proof mining} program, as it is known today, has generally focused on developing general metatheorems that guarantee the extractability of effective uniform bounds from proofs of theorems of a specific logical form - usually relatively simple $\Pi_2$ theorems for which direct computational data can be found. In other words, on the whole proof interpretations have been applied to extract purely quantitative information from a fairly restricted class of theorems. However, the last decade has seen proof interpretations employed much more widely, with an increasing emphasis on understanding the qualitative aspects of interpreted proofs.

There are two main reasons for this. The first is a greater appreciation of the mathematical significance of proof interpretations. It was recently observed (e.g. \cite{Kohlenbach(2008)}) that the monotone variant of G\"{o}del's dialectica interpretation is closely related to the so-called `correspondence principle' between finite and infinite dynamical systems as discussed by T. Tao in \cite[Ch. 1.3]{Tao(2008A)}. This observation lies behind current applications of the dialectica interpretation in ergodic theory (see e.g. Avigad \cite{Avigad(2009)}), which in particular explore the dialectica interpretation of Cauchy convergence, known to mathematicians as \textit{metastability}.

The second reason is an improved understanding of the \textit{semantics} of proof interpretations. Formal translations on proofs are highly syntactic and in particular the functional interpretation of proofs that make use of full arithmetic comprehension traditionally involves Spector's abstruse bar recursion schema. Consequently, realisers for interpreted proofs are often stated as almost unreadable higher type terms. This issue is addressed in recent work by the authors and M. Escard{\'o} \cite{EO(2010A),EO(2011A)}, who show that the product of selection functions provides us with an intuitive alternative to bar recursion that can be understood in terms of the computation of optimal strategies in a certain class of sequential games. This makes it easier to appreciate the operational behaviour of realisers of interpreted theorems in analysis. 

Therefore the authors believe that it is both practical and meaningful to apply proof interpretations to classical proofs with the object of producing a \textit{mathematical} proof of a new, finitary theorem, as opposed to just extracting a new piece of quantitative information. In this article we apply G\"{o}del's dialectica interpretation to Erd\H{o}s and Rado's proof of Ramsey's theorem for pairs, similarly to what has been done in \cite{Kreuzer(2009),Kreuzer(2009A)}. Our main aim here, however, is to produce an intuitive combinatorial proof of the finitary form of the Ramsey's theorem given in Section \ref{sec-intro-ram}. For that purpose, we endeavour to strip our proof of heavy logical syntax in order to understand it in mathematical terms. In a broader sense our aim is to portray the dialectica interpretation as an intelligent translation on proofs as opposed to just a syntactic translation on logical formulas.

The paper is organised as follows. We begin by formulating Ramsey's theorem and its proof in the language of formal arithmetic. We then briefly discuss the main building block of our extracted proof, the product of selection functions, and in Section \ref{sec-extract} we prove our finitary version of the theorem. Finally, we discuss a game theoretic reading of our proof.

\subsection{Preliminaries}

In this article we assume that the reader is familiar with G\"{o}del's \textit{dialectica interpretation} of classical proofs (cf. \cite{Avigad(98),Kohlenbach(2008)} and the original paper \cite{Goedel(58)}), by which we implicitly mean G\"{o}del's dialectica interpretation combined as usual with the negative translation\footnote{As in \cite{Kohlenbach(2008)} we adopt Kuroda's variant of the negative translation.}. We do not assume familiarity with the authors' recent work on the product of selection functions - although the reader is encouraged to consult \cite{EO(2011A)} for a more detailed treatment of the results mentioned in Section \ref{sec-selection}.


The theory $\PAomega$ is Peano arithmetic in all finite types, and $\systemT$ is G\"{o}del's quantifier-free theory of higher-type primitive recursive functionals (see \cite{Avigad(98)} for full definitions). We make informal use of types like the Booleans $\BB = \{0,1\}$ and finite sequence types $X^\ast$. \\[2mm]
{\bf Notation}. We use the following abbreviations:
\begin{description}
\item $s\ast t$ is the concatenation of sequences $s$ and $t$.
\item $\ext{s} \equiv s * \cZero^{X^\NN}$ a canonical infinite extension of a finite sequence $s$.
\item $\initSeg{\alpha^{X^\NN}}{n}\equiv\pair{\alpha 0,\ldots,\alpha (n-1)}$ is the initial segment of $\alpha$ of length $n$.
\end{description}

Finally, we make use of the following key logical principles. $\Pi_1^0$ \emph{countable choice} is given by the schema 
\begin{equation*}\Pi_1^0\mbox{-}\CA \ \colon \ \forall n\exists x^X\forall y^Y A_n(x,y)\to\exists\alpha^{\NN\to X}\forall n,y A_n(\alpha n,y),\end{equation*}
where the $A_n$ are quantifier-free, \emph{weak K\"{o}nig's lemma} is the statement that any infinite decidable binary tree $T$ has an infinite branch: 
\begin{equation*}\WKL \ \colon \ \forall n\exists s^{\BB^\ast}(|s|  \wedge T(s))\to\exists\alpha^{\NN\to \BB}\forall n T(\initSeg{\alpha}{n}), \end{equation*}
and the \emph{infinite pigeonhole principle} states that for any $n$-colouring $c$ of the natural numbers, at least one colour $x$ is used infinitely often:
\begin{equation*} 
\IPP \ \colon \ \forall c^{\NN\to [n]}\exists x,p^{\NN\to\NN}\forall k (pk \geq k \wedge c(pk)=x).
\end{equation*}
Note that of these only $\IPP$ is provable in $\PAomega$

\subsection{Ramsey's theorem for pairs}
\label{sec-intro-ram}

In this article we only consider Ramsey's theorem for pairs and $2$-colourings on the basis that our results can be extended to the more general theorem, although in the course of our program extraction we hint at how key steps can be generalised for the $n$-colour case. 

Let $[\NN]^2$ denote the set of subsets of $\NN$ of size two, and suppose we are given a colouring $c \colon [\NN]^2 \to \BB$ of $[\NN]^2$ with two colours. Ramsey's theorem says that for any such colouring there exists an infinite pairwise monochromatic subset of $\NN$ i.e. an infinite set $S \subseteq \NN$ such that all elements of $[S]^2$ have colour $x$ for some $x\in\BB$. Formally, we write Ramsey's theorem as
\begin{equation*}\RAMPAIR{c} \ \colon \ \exists x^\BB \exists F^{\NN\to\NN} \forall k (F k \geq k \wedge \forall i, j \leq k (F i < F j \to c(\{F i, F j\}) = x)). \end{equation*}
Here the infinite monochromatic set is encoded by the function $F$ and is given by $S_F=\{Fk \ \colon \ k\in\NN\}$. Our main result is a constructive proof of the dialetica interpretation of $\RAMPAIR{c}$:
\begin{maintheorem}Suppose the colouring $c$ is fixed. For any functional $\eta\colon \BB\times \NN^\NN\to\NN$ there exists a colour $x \colon \BB$ and a function $F\colon \NN\to\NN$ satisfying
\begin{equation} \label{finramsey}
\forall k \leq \eta_x F (F k \geq k \wedge \forall i, j \leq k (F i < F j \to c(\{F i, F j\}) = x)).
\end{equation}
\end{maintheorem}
As with all $\Sigma_2$ theorems, the functional interpretation of Ramsey's theorem coincides with Kreisel's \textit{no-counterexample} interpretation (n.c.i.). The intuition is that the counterexample functions $\eta_0$ and $\eta_1$ attempt to show that for any $F$ the set $S_F$ cannot be pairwise monochromatic, and we are challenged to effectively refute any such counterexample functions. Alternatively, we can view $\eta$ as a function that specifies in advance how we want to ``use" Ramsey's theorem in a specific computation: while in general there is no effective way of realising Ramsey's theorem, given $\eta$ we can produce an \textit{approximation} to a monochromatic set that is sufficient for the computation we have in mind.

\subsection{Comparison to existing work}
\label{subsec-related}

Ramsey's theorem has been extensively studied in logic, so it is important to outline how our work contrasts to related papers on the constructive content of the theorem. 

In \cite{Bellin(90)} Bellin uses proof theoretic techniques to produce a proof of a finitary version of Ramseys' theorem similar to (\ref{finramsey}). However, his proof differs from ours in two important respects. Firstly, his is based on Ramsey's original proof as opposed to the one analysed here by Erd\H{o}s and Rado, and secondly he uses cut-elimination and Herbrand's theorem as opposed to the dialectica interpretation.  

A formalisation of Erd\H{o}s and Rado's proof was recently given by Kreuzer and Kohlenbach in \cite{Kreuzer(2009)}, and a bar-recursive realizer for its functional interpretation was stated in \cite{Kreuzer(2009A)}. The main achievement of these works is to calibrate the proof theoretic strength of $\RAMPAIR{c}$ and establish its contribution to the complexity of extracted programs in certain cases, whereas our goal is to produce an intuitive constructive version of the Erd\H{o}s-Rado proof using the product of selection functions that can be understood in mathematical terms. 

We note that while our formalisation of the Erd\H{o}s-Rado proof is influenced by theirs in that we also encode the Erd\H{o}s-Rado min-monochromatic tree as as a binary $\Sigma^0_1$ tree, our treatment differs substantially from \cite{Kreuzer(2009),Kreuzer(2009A)}. In particular we encode min-monochromatic branches using a different $\Sigma^0_1$ tree, and in our program extraction we use new interpretations of $\WKL$ and $\pCA$ using the product of selection functions, as opposed to the standard bar-recursive interpretations of Howard and Spector used in \cite{Kreuzer(2009A)}.

Veldman and Bezem \cite{Veldman(1993)} discovered an interesting constructive variant of Ramsey's theorem. That formulation and proof have been simplified by Coquand in \cite{Coquand(1994B)} (see also \cite{Coquand(1994)}). Coquand's proof makes use a recursion on well-founded trees similar to Spector's bar recursion. The main difference being that in our algorithm the well-founded tree is not given explicitly as part of the problem, as it is in Coquand's formulation of Ramsey's theorem.

To summarise, then, in comparison to existing work our analysis of Ramsey's theorem combines the following key benefits:
\begin{enumerate}

\item Our \textbf{constructive interpretation} of the theorem is based on G\"{o}del's dialectica interpretation. The advantage of this is that our theorem is more `computational' than e.g. that of Veldman and Bezem \cite{Veldman(1993)} in that we explicitly prove the existence of arbitrarily large approximations to a monochromatic set. Moreover, as indicated previously, our finitary Ramsey's theorem can be related to the finitisations of infinitary theorems in the sense of Tao \cite{Tao(2008A)}. 

\item Our \textbf{constructive proof} of Ramsey's theorem is based on the product of selection functions as opposed to Spector's bar recursion and can be given a clear game theoretic interpretation (Section \ref{subsec-gameint}).

\end{enumerate}

\section{A Formal Proof of Ramsey's Theorem}
\label{sec-formal}

\noindent\textbf{Notation.} For simplicity we encode a colouring $c\colon [\NN]^2\to\BB$ as a map $c\colon \NN^2\to\BB$ with the property that $c(i,j)=c(j,i)$ for all $i,j$. \\

We now present a formal proof of Ramsey's theorem based on that of Erd\H{o}s and Rado (\cite{Erdos(84)}, Section 10.2). In doing so we show that $\RAMPAIR{c}$ can be formalised in $\PAomega+\WKL+\pCA$, and therefore its functional interpretation can theoretically be witnessed using Spector's bar recursion \cite{Kohlenbach(2008),Spector(62)}, or alternatively (as we demonstrate in Section \ref{sec-selection}) the product of selection functions. Of course, actually constructing this witness is non-trivial -- the soundness theorem for the dialectica interpretation gives a syntactic algorithm which would be impractical to carry out by hand. Therefore, we make use of the soundness theorem as a very rough guide on how to proceed but use shortcuts whenever possible. 

The main idea behind the classical proof is, given a colouring $c$, to organise the natural numbers into a tree (described as an ordering $\prec$ on $\NN$) whose branches are \textit{min-monochromatic}, in the sense that $c(i,j)=c(i,k)$ for $i\prec j\prec k$, where $i \prec j$ says that node $i$ precedes $j$ in the tree. This is the so-called Erd\H{o}s-Rado (E/R) tree. By K\"{o}nig's lemma the E/R tree has an infinite min-monochromatic branch $a\colon \NN^\NN$, so by the infinite pigeonhole principle applied to the colouring $c^a(i)=c(a(i),a(i+1))$ there exists an infinite subset of the branch that is pairwise monochromatic.

Our formal proof proceeds, in a similar fashion to \cite{Kreuzer(2009)}, as follows. We encode branches of the E/R tree by an infinite $\Sigma^0_1$ binary tree $T$ (Definition \ref{bin-ertree}). We then reduce $T$ to an infinite \emph{decidable} binary tree using $\pCA$ (Lemma \ref{skolem-lemma}), which by $\WKL$ has an infinite branch. We then show that an infinite branch of $T$ does indeed encode an infinite branch of the E/R tree (Lemma \ref{build-a}). Hence, we are finally able to complete the proof using $\IPP$ (Theorem \ref{thm-ramsey}). Because we are only considering here the case of two colours, we do not need the full $\IPP$ but only a very simple instance of it (case $n = 2$). Nevertheless, we discuss the whole construction in terms of the full $\IPP$ so that a generalisation to the case of finitely many colours is more straightforward. We sketch our formal proof in Figure \ref{fig-prooftree}. Here $\ER{c}$ abbreviates the statement that the E/R tree defined by $c$ has an infinite branch.  

\begin{figure}[h]
\begin{center}
\begin{prooftree}
\AxiomC{$\IPP$}
\AxiomC{$\WKL$}
\AxiomC{$\pCA$}
\RightLabel{\scriptsize{\ref{skolem-lemma}-\ref{erbranch}}}
\BinaryInfC{$\ER{c}$}
\RightLabel{\scriptsize{\ref{thm-ramsey}}}
\BinaryInfC{$\RAMPAIR{c}$}
\end{prooftree}
\end{center}
\caption{Formal proof of Ramsey's theorem.}
\label{fig-prooftree}
\end{figure}

It is important to remark why we have chosen this proof over Ramsey's seemingly simpler proof in \cite{Ramsey(30)}. Ramsey constructs an infinite min-monochromatic branch directly using dependent choice: First we define $a0=0$, then we use $\IPP$ to produce an infinite set $A_1\subseteq\NN\backslash {0}$ that is monochromatic under $c_{0}(i)=c(0,i)$ and define $a(1)=\min A_0$. Next use $\IPP$ to produce an infinite set $A_2\subseteq A_1\backslash {a(1)}$ that is monochromatic under $c_{a(1)}(i)=c(a(1),i)$ and define $a(2)=\min A_1$ and so on. It is easy to see that the resulting $a$ is min-monochromatic. However, Ramsey's construction uses dependent choice of type $1$ (Simpson shows  in \cite{Simpson(99)} that it cannot be formalised in the subsystem $\ACA$), therefore its computational interpretation would seemingly involve bar recursion/product of selection functions of level $1$. Our interpretation of the Erd\H{o}s-Rado proof, on the other hand, makes use of the product of selection functions of lowest type only, meaning that our construction is computationally simpler.

\begin{definition}[Erd\H{o}s/Rado tree] Given a colouring $c \colon \NN^2 \to \BB$, define a partial order $\prec$ on $\NN$ recursively as follows:
\begin{enumerate}
\item $0\prec 1$
\item Given that $\prec$ is already defined on the initial segment of the natural numbers $\is{j}$, for $j<i$ define
	\[ j \prec i \quad \mbox{iff} \quad c(k, i)=c(k, j), \ \mbox{for all } k \prec j \]
\end{enumerate}
\end{definition}

It is easy to show that $\prec$ defines a tree on $\NN$, the so-called Erd\H{o}s/Rado tree, and that its branches are min-monochromatic i.e. $c(k,i)=c(k,j)$ for $k\prec i\prec j$. Moreover, the tree is binary branching because $i$ and $j$ are successors of $k$ if and only if $c(k, i) \neq c(k, j)$. For proofs of these facts see \cite[Section 4]{Kreuzer(2009)}. We consider the following $\Sigma^0_1$ tree.  

\begin{definition}[Binary Erd\H{o}s/Rado tree]\label{bin-ertree} Define the $\Sigma^0_1$-predicate $T$ on $\BB^*$ by
\eqleft{T(s) := \exists k (\underbrace{\exists k' \!\in\! [|s|, k] \, \forall i < |s| (s_i = 0 \;\leftrightarrow\; i \prec k')}_{\mt(s, k)}).}
\end{definition}

A 0-1 sequence $s$ belongs to $T$ if it is the characteristic function of a finite branch of the Erd\H{o}s/Rado tree. We use a $k$ and a $k'$ in order to make $T(s)$ a $\Sigma^0_1$-predicate \emph{monotone} on unbounded quantifier $k$. This will simplify the construction.

\begin{lemma} The following are simple properties of $T$
\begin{itemize}
\item[$(i)$] $T$ as defined above is an infinite tree. 
\item[$(ii)$] The branches of $T$ are characteristic functions of branches of the E/R tree.
\item[$(iii)$] $T$ satisfies the following monotonicity conditions\footnote{It will become clear in Section \ref{sec-extract} why we require our tree to have these properties.}:
\begin{equation*} \label{bw-hypothesis}
(\mona) \;\; \mt(s * t, k) \to \mt(s, k) \quad\quad \mbox{and} \quad\quad (\monb) \;\; \mt(s, k) \to \mt(s, k+l).
\end{equation*}
\end{itemize}
\end{lemma}
\begin{proof} $(i)$ Clearly $T$ is prefix closed. Moreover, for all $n$, $T(s)$ has a branch $s$ of length $n$ given by $s_i = 0$ iff $i \prec n$, for $i < n$. $(ii)$ $T(s)$ implies that the set defined by $s$ is an initial segment of the branch of the predecessors of $k'$, denoted $pd(k')$, of the E/R tree. Therefore is also a branch of the E/R tree. $(iii)$ $(\mona)$ is obvious, and $(\monb)$ follows because we only ask for a \emph{bound} on $k'$. 
\end{proof}

The first step in our proof is to prove the existence of a function $\beta$ which will allow us to turn the $\Sigma^0_1$-tree $T(s)$ into a decidable tree.

\begin{lemma}[Monotone Skolem function]\label{skolem-lemma} There exists a function $\beta$ such that
\begin{equation} \label{skolem-function-beta}
\forall n \forall s (|s| = n \wedge \exists k \mt(s, k) \to \mt(s, \beta n)).
\end{equation}
\end{lemma}
\begin{proof} Classically we have that
\[ \forall n \forall s (|s| = n \to \exists k'(\exists k \mt(s, k) \to \mt(s, k'))). \]
By \emph{bounded collection} and monotonicity of $\mt$ we have
\[ \forall n \exists k' \forall s (|s| = n \wedge (\exists k \mt(s, k) \to \mt(s, k'))). \]
Finally, by \emph{countable choice} for $\Pi^0_1$-formulas we obtain the function $\beta$.
\end{proof}

\vspace{3mm}

The Skolem function $\beta$ allows us to turn the $\Sigma^0_1$-predicate $T(s)$ into a decidable predicate:

\begin{corollary} Given a function $\beta$ satisfying (\ref{skolem-function-beta}) we have that $T(s)$ is equivalent to
\[ \underbrace{\exists k \!\in\! [|s|, \beta(|s|)] \, \forall i < |s| (s_i = 0 \;\leftrightarrow\; i \prec k)}_{T^\beta(s)}. \]
\end{corollary}

Once we have a decidable infinite finitely branching tree $T^\beta(s)$ we can apply weak K\"onig's lemma to obtain an infinite path in the tree.

\begin{lemma} \label{wkl-used} There exists an infinite sequence $\alpha$ such that 
\begin{equation} \label{wkl-result}
\forall n \underbrace{\exists k \!\in\! [n, \beta n] \, \forall i < n (\alpha(i) = 0 \;\leftrightarrow\; i \prec k)}_{T^\beta(\initSeg{\alpha}{n})}.
\end{equation}
\end{lemma}
\begin{proof} By \emph{weak K\"onig's lemma}.
\end{proof}

However, it remains to show that an infinite branch of $T$ encodes an infinite branch of the Erd\H{o}s/Rado tree.

\begin{lemma} \label{build-a} The sequence $\alpha$ has infinitely many zeros, i.e. it is the characteristic function of an infinite set. More specifically, we can construct a function $a\colon\NN\to\NN$ that returns the first $k\geq n$ with $\alpha(k)=0$.
\end{lemma}
\begin{proof} Define $a(n)$ as
\eqleft{a(n) =
\left\{
\begin{array}{ll}
	0 & \mbox{if} \; n = 0 \\[2mm]
	k & \mbox{for least $k \in [n, \beta(\beta n + 1)]$ such that $\alpha(k) = 0$}.
\end{array}
\right.
}
We show that $a$ is well-defined, so that in fact $\alpha(a(n)) = 0$ for all $n$. Because the image of $a$ is unbounded the result follows. We must have that $\alpha(0)=0$ by definition of $\prec$. Now given $n > 0$, let $i < n$ be the largest such that $\alpha(i) = 0$. Consider $k \in [n, \beta n]$ which by (\ref{wkl-result}) satisfies $\forall i < n (\alpha(i) = 0 \;\leftrightarrow\; i \prec k)$; and hence $i \prec k$. Now, let $n$ be $\beta n + 1$ in (\ref{wkl-result}) so that we have a $k' \in [\beta n + 1, \beta(\beta n + 1)]$ satisfying $\forall i < n (\alpha(i) = 0 \;\leftrightarrow\; i \prec k')$; and hence $i \prec k'$ as well. Finally, let $n$ be $\beta(\beta n + 1) + 1$ in (\ref{wkl-result}) so that we have a $k'' \in [\beta(\beta n + 1) + 1, \beta(\beta(\beta n + 1) + 1)]$ satisfying $\forall i < n (\alpha(i) = 0 \;\leftrightarrow\; i \prec k'')$; so that also $i \prec k''$. Since we have $i \prec k$ and $i \prec k'$ and $i \prec k''$, it follows that either $k \prec k'$ or $k \prec k''$ or $k' \prec k''$, since the Erd\H{o}s/Rado tree is binary branching. Hence, either $\alpha(k) = 0$ or $\alpha(k') = 0$, and either way there is some $l\in [n,\beta(\beta n + 1)]$ with $\alpha(l)=0$.\end{proof}

\begin{remark}\label{build-a-remark}Note that in verifying that $\alpha(a(n))=0$ we have only used (\ref{wkl-result}) up to the point $\max\{n,\beta n+1,\beta(\beta n+1)+1\}$. We use this fact later to show that a sufficiently large approximation of $\alpha$ is sufficient for an approximation of $a$. \end{remark}

\begin{remark} \label{remark-ncolours}For the $n$ colour case the Erd\"{o}s/Rado tree is still \emph{finitely} branching but not \emph{binary} branching as it is for case of two colours $n=2$. This in particular means that generalising the proof of Lemma \ref{build-a} for arbitrarily many colours is non-trivial (although still routine), and the construction of the function $a$ and the bound in Remark \ref{build-a-remark} are more complex (involving further iterations of $\beta$). Note, however, that the tree $T(s)$ would still be \emph{binary} branching, even in the case of $n$ colours, as $T(s)$ means that $s$ is the ``characteristic function" of a branch in the Erd\"{o}s/Rado tree. In particular, only the weak form of K\"onig's lemma is required in the general case as well. \end{remark}

\begin{corollary}\label{erbranch}There exists an infinite set that is min-monochromatic under the colouring $c \colon \NN^2 \to \BB$.\end{corollary}

\begin{proof}Clearly the set $\{an \ \colon \ n\in\NN\}$ is infinite. Moreover for $ak<ai<aj$ it follows from (\ref{wkl-result}) for $n=aj+1$ that $ak\prec ai\prec aj$, and therefore $c(ak,ai)=c(ak,aj)$.\end{proof}

All that remains is to apply the infinite pigeonhole principle to the min-monochromatic branch given by $a$.

\begin{theorem}[Ramsey's theorem]\label{thm-ramsey} For every colouring $c \colon \NN^2 \to \BB$
\[ \exists x^{\BB} \exists F^{\NN\to\NN} \forall k (F k \geq k \wedge \forall i, j \leq k (Fi < Fj \to c(Fi,Fj)=x)). \]
\end{theorem}
\begin{proof} Let $a$ be as in the previous lemma. Clearly $a n \geq n$, so the image of $a$ is an infinite set. Moreover,
\begin{equation} \label{a-mon}
c(a(k), a(i)) = c(a(k), a(j)),
\end{equation}
whenever $a(k) < a(i)$ and $a(k) < a(j)$, by (\ref{wkl-result}) and definition of $a$. Finally, define a couloring $c' \colon \NN \to \BB$ as $c'(n) = c(a(n), a(a(n)+1))$. By the \emph{infinite pigeon-hole principle} we have a $p$ and an $x$ such that $p(n) \geq n$ and
\[ x = c'(p i) = c(a(p i), a(a(pi) + 1)) \stackrel{(\ref{a-mon})}{=} c(a(p i), a(p j)),  \]
for $a(p i) < a(p j)$. Hence, $F(i) = a(p i)$ does the job.
\end{proof}

\section{The Product of Selection Functions}
\label{sec-selection}

It is well known that just as Peano arithmetic has a dialectica interpretation in the primitive recursive functionals of finite type $\systemT$, classical analysis (i.e. $\PAomega+\CA$) has a dialectica interpretation in the bar recursive functionals $\systemT+\BR$, where $\BR$ is the bar recursor introduced by Spector in his fundamental paper \cite{Spector(62)}.

Spector's bar recursion is rather abstruse, and the operational behaviour of programs that make use of this kind of recursion tends to be quite difficult to understand. This was not originally an issue, as Spector's aim was simply to obtain a relative consistency proof for analysis. However, when using the dialectica interpretation to extract programs from proofs in analysis, it is sensible to ask whether there is a more intuitive alternative to bar recursion that facilitates a better understanding of these programs.

In \cite{EO(2010A)}, the first author and Escard\'{o} propose the \emph{product of selection functions} as a (computationally equivalent) alternative to bar recursion. In contrast to bar recursion, the product of selection functions is a versatile construction that seems to appear naturally in a variety of different contexts in mathematics and computer science, such as fixed point theory, algorithms and game theory. As such, extracted programs that make use of the product tend to be more illuminating.

In this section we briefly outline the main results that will be used in Section \ref{sec-extract}, and provide some motivation as to why we prefer the product over bar recursion. The reader is encouraged to consult the survey paper \cite{EO(2011A)} and a recent paper on the extraction of programs from proofs using selection functions \cite{OP(2012)} for further details and discussion.

We call \emph{selection function} any element of type $J_RX:=(X\to R)\to X$. Given a selection function $\varepsilon \colon (X \to R) \to X$ we denote by $\selEmb{\varepsilon} \colon (X \to R) \to R$ the functional $\selEmb{\varepsilon}(p) \stackrel{R}{=} p(\varepsilon p)$.
%





\begin{definition}[Binary product of selection functions \cite{EO(2009)}] Given a selection function $\varepsilon \colon J_RX$ and family of selection functions $\delta_x  \colon J_RY$ and a predicate $q\colon X\times Y\to R$, let
\begin{align*} B[x^X] &\stackrel{Y}{:=} \delta(x, \lambda y.q(x,y)) \\ a &\stackrel{X}{:=} \varepsilon(\lambda x.q(x,B[x])).\end{align*}
The binary product $\varepsilon\otimes\delta$ of the selection functions $\varepsilon$ and $\delta$ is another selection function, of type $J_R(X\times Y)$, defined by
\[(\varepsilon\otimes\delta)(q) \stackrel{X\times Y}{:=} \pair{a,B[a]}.\]
\end{definition}

As described in \cite{EO(2009)}, we can iterate the binary product of selection functions an unbounded number of times, where the length of the iteration is dependent on the output of the product in the following sense.

\begin{definition}[Iterated product of selection functions \cite{EO(2009)}] \label{unbounded} Suppose we are given a family of selection functions $(\varepsilon_s\colon J_R X)$, where $s \colon X^*$. The \textit{explicitly controlled unbounded product} of the selection functions $\varepsilon_s$ is defined by the recursion schema
\begin{equation} \label{Prod}
\EPS{s}{\omega}{\varepsilon}(q) \stackrel{X^{\NN}}{=}
\left\{
\begin{array}{ll}
	{\bf 0} & {\rm if } \; \omega(\ext{s})<|s| \\[2mm]
	(\varepsilon_s \otimes \lambda x . \EPS{s * x}{\omega}{\varepsilon})(q) & {\rm otherwise}
\end{array}
\right.
\end{equation}
where $s \colon X^*$, $q \colon X^{\NN} \to R$ and $\omega \colon X^{\NN} \to \NN$.
\end{definition}

When $\omega$ is a constant function, say $\omega \alpha = n$, this corresponds to a finite iteration of the binary product. The functional $\omega$ acts as a control, terminating the recursion once it has produced a sequence $s$ satisfying $\omega(\ext{s})<|s|$. By simply unwinding the definition of the binary product in (\ref{Prod}) we obtain an equivalent equation
\begin{equation} \label{Prod-var}
\EPS{s}{\omega}{\varepsilon}(q) \stackrel{X^{\NN}}{=}
\left\{
\begin{array}{ll}
	{\bf 0} & {\rm if } \; \omega(\ext{s})<|s| \\[2mm]
	a_s * \EPS{s * a_s}{\omega}{\varepsilon}(q_{a_s}) & {\rm otherwise}
\end{array}
\right.
\end{equation}
where $a_s = \varepsilon_s(\lambda x . \selEmb{\EPS{s * x}{\omega}{\varepsilon}}(q_x))$, with $q_x(\alpha) = q(x * \alpha)$ and $\selEmb{\delta}(p) \stackrel{R}{=} p(\delta p)$. 

For fixed $\omega, \varepsilon$ and $q$ we should think of $\EPS{s}{\omega}{\varepsilon}(q)$ as computing an infinite extension to any given finite sequence $s$. Hence, we are interested in the sequence $s * \EPS{s}{\omega}{\varepsilon}(q)$. The fundamental property of $\EPSs$ is that the infinite extension of an initial segment $\initSeg{\alpha}{n}$ of a previous infinite extension $\alpha$ is identical to the original infinite extension. Formally:

\begin{lemma}[Main lemma on $\EPSs$] \label{main-lemma} If $\alpha = \EPS{s}{\omega}{\varepsilon}(q)$ then, for all $n$, 
\begin{equation} \label{main-lemma-eq}
\alpha = \initSeg{\alpha}{n} * \EPS{s * \initSeg{\alpha}{n}}{\omega}{\varepsilon}(q_{\initSeg{\alpha}{n}}).
\end{equation}
\end{lemma}
\begin{proof} Induction on $n$. See \cite{EO(2011A)} for details.\end{proof}

This lemma is the main building block behind the proof of the following fundamental theorem about $\EPSs$.

\begin{theorem}[Main theorem on $\EPSs$] \label{eps-main} Let $q \colon X^\NN \to R$ and $\omega \colon X^\NN \to \NN$ and $\varepsilon_s \colon J_R X$ be given. Define
\eqleft{
\begin{array}{lcl}
\alpha & \stackrel{X^\NN}{=} & \EPS{\emptyseq}{\omega}{\varepsilon}(q) \\[2mm]
p_s(x) & \stackrel{R}{=} & \selEmb{\EPS{s * x}{\omega}{\varepsilon}}(q_{s * x}).
\end{array}
}
For $n \leq \omega(\alpha)$ we have
\[
\begin{array}{lcl}
	\alpha(n) & \stackrel{X}{=} & \varepsilon_{\initSeg{\alpha}{n}}(p_{\initSeg{\alpha}{n}}) \\[2mm]
	q \alpha & \stackrel{R}{=} & \selEmb{\varepsilon_{\initSeg{\alpha}{n}}}(p_{\initSeg{\alpha}{n}}).
\end{array}
\]
\end{theorem}
\begin{proof} Assume $n \leq \omega(\alpha)$. We argue that $(*) \; n \leq \omega(\initSeg{\alpha}{n} * \cZero)$. Otherwise, assuming $n > \omega(\initSeg{\alpha}{n} * \cZero)$ we would have, by Lemma \ref{main-lemma}, that $\alpha = \initSeg{\alpha}{n} * \cZero$. And hence, $n > \omega(\initSeg{\alpha}{n} * \cZero) = \omega(\alpha) \geq n$, which is a contradiction. Hence, we have that
\eqleft{
\begin{array}{lcl}
	\alpha(n) 
		& \stackrel{\textup{L}\ref{main-lemma}}{=} & \EPS{\initSeg{\alpha}{n}}{\omega}{\varepsilon}(q_{\initSeg{\alpha}{n}})(0) \\[1mm]
		& \stackrel{(*)}{=} & \varepsilon_{\initSeg{\alpha}{n}}(\lambda x . \selEmb{\EPS{\initSeg{\alpha}{n} * x}{\omega}{\varepsilon}}(q_{\initSeg{\alpha}{n} * x})) \\[2mm]
		& = & \varepsilon_{\initSeg{\alpha}{n}}(p_{\initSeg{\alpha}{n}}).
\end{array}
}
For the second identity, we have
\eqleft{
\begin{array}{lcl}
	q \alpha & \stackrel{\textup{L}\ref{main-lemma}}{=} & q_{\initSeg{\alpha}{n+1}}(\EPS{\initSeg{\alpha_{\emptyseq}}{n+1}}{\omega}{\varepsilon}(q_{\initSeg{\alpha}{n+1}})) \\[2mm]
	& = & p_{\initSeg{\alpha}{n}}(\alpha(n)) \\[2mm]
	& = & \selEmb{\varepsilon_{\initSeg{\alpha}{n}}}(p_{\initSeg{\alpha}{n}}),
\end{array}
}
where the last equality uses that $\alpha(n) = \varepsilon_{\initSeg{\alpha}{n}}(p_{\initSeg{\alpha}{n}})$ is already shown.
\end{proof}

The significance of Theorem \ref{eps-main} is that it shows how the product of selection functions computes a sequence $\alpha$ that represents, in some sense, a sequential equilibrium between the selection functions up to the point $\omega\alpha$. This kind of equilibrium appears in a variety of different contexts, most notably the following. 

\subsection{Optimal strategies in sequential games}
\label{subsec-games}

As discussed in \cite{EO(2011A)}, the parameters $\varepsilon$, $q$ and $\omega$ of $\EPSs$ naturally define a sequential game
\eqleft{\game^{X,R} = (\varepsilon,q,\omega)}
of type $(X,R)$. We imagine $X$ as a set of possible moves at each round, and $R$ as a set of possible outcomes. A finite sequence $s\colon X^\ast$ can be thought of as a position in the game determined by the first $|s|$ moves, while an infinite sequence $\alpha\colon X^\NN$ can be thought of as a play of the game. We then make the following associations:

\begin{itemize}

\item $\varepsilon_s\colon J_RX$ determines an optimal move at position $s$ given that the outcome of each possible move $X\to R$ is known.

\item $q\colon X^\NN\to R$ determines the outcome of each play $\alpha$. 

\item $\omega\colon X^\NN\to \NN$ determines the relevant part of a play. A position $s$ is relevant if $\omega\hat s\geq |s|$.

\end{itemize}

We refer to $q$ and $\omega$ as the outcome function and control function, respectively. In general these games can be thought of as \textit{unbounded} games in which we only care about a finite initial segment of any play, as determined by $\omega$. In the context of game theory Theorem \ref{eps-main} can be rephrased as the following.

\begin{theorem}\label{eps-games} The sequence $\alpha=\EPS{\pair{}}{\omega}{\varepsilon}(q)$ is an \textit{optimal play} in the game $\game^{X, R} = (\varepsilon, q, \omega)$.\end{theorem} 

We do not go into details on exactly what constitutes an optimal play, or how Theorem \ref{eps-games} is proved (for this see \cite{EO(2010A)}) but the main idea is not difficult to see. We imagine the function $p_s$ defined in Theorem \ref{eps-main} as giving outcome of playing $x$ at position $s$, under the assumption that all subsequent moves are played optimally, and thus $\varepsilon_s(p_s)$ is the best move at position $s$. 

The product of selection functions carries out a backtracking algorithm and eventually computes a sequence $\alpha$ such that $\alpha(n)=\varepsilon_{\initSeg{\alpha}{n}}(p_{\initSeg{\alpha}{n}})$ for all $n\leq\omega\alpha$. In other words $\alpha(0)$ is the best move at position $\pair{}$, $\alpha(1)$ the best move at position $\pair{\alpha(0)}$ and more generally $\alpha(n)$ the best move at position $\initSeg{\alpha}{n}$ for as long as $\initSeg{\alpha}{n}$ is relevant. In this sense $\alpha$ forms an optimal play of $\game$. We remark that the strategy profile arising from this notion of optimal play coincides with the \textit{Nash equilibrium} of a sequential game (see \cite{EO(2012A)}).

\subsection{The dialectica interpretation of the axiom of choice}
\label{subsec-choice}

Sequential games provide us with perhaps the most illuminating instance of the equilibrium computed by the product of selection functions. Remarkably, another instance is the dialectica interpretation of the axiom of choice.

The functional interpretation of $\pCA$ is equivalent to
\begin{equation*}\forall\varepsilon,q,\omega(\forall n,p A_n(\varepsilon_np,p(\varepsilon_np))\to\exists\alpha\forall n\leq\omega\alpha A_n(\alpha n,q\alpha)).\end{equation*}
It challenges us, given a collection of `pointwise' strategies $\varepsilon_n$ that witness the no-counterexample interpretation of the $A_n$, to combine them into a global strategy $\alpha$ that witnesses the n.c.i. of $\forall n A_n$. It is clear by Theorem \ref{eps-main} that the product of selection functions does the job.
\begin{theorem}\label{eps-choice}The functional $\lambda\varepsilon,q,\omega.\EPS{\pair{}}{\omega}{\varepsilon}(q)$ realises the dialectica interpretation of $\pCA$.\end{theorem}

Again, we do not go into detail, this time we refer the reader to \cite{EO(2010A),OP(2012)}. It can be shown more generally that $\EPSs$ directly witnesses the dialectica interpretation of dependent choice for arbitrary formulas, and that a finite form of $\EPSs$ with $\omega$ constant directly interprets finite choice or \emph{bounded collection}. Moreover, $\EPSs$ is primitive recursively equivalent to Spector's bar recursion \cite{EO(2009)}, and its finite form is equivalent to primitive recursion over a weak base theory \cite{EOP(2011)}.

The key point we emphasise is that, as a computational analogue of choice, the product of selection functions is an extremely useful recursion schema to have at our disposal when it comes to extracting programs from proofs in both arithmetic and analysis. The fact that it also computes optimal strategies in sequential games means that extracted programs can be given an intuitive game-theoretic semantics, in the sense that we can often informally identify the ``classical" dialectica interpretation $A^{ND}$ of a theorem $A$ with a partially defined sequential game:
\[ A^{ND} \quad \sim \quad \game_A,
\]
where a realizer for $A^{ND}$ can given in terms of optimal strategies in $\game_A$. This gives the product of selection functions a clear advantage over bar recursion when interpreting theorems in analysis.

We now extract a program from the formal proof of Ramsey's theorem described in Section \ref{sec-formal} using the product of selection functions $\EPSs$. We apply the product directly, appealing only to the main Theorem \ref{eps-main}. The other results in this section were mentioned simply to provide some motivation as to why the product appears naturally in proof theory and why it is preferred to the more traditional modes of recursion.

\section{A Constructive Proof of Ramsey's Theorem}
\label{sec-extract}

Before launching into the full interpretation of the classical proof, it is instructive to look at the overall structure of our extracted program. Let us first look at the computational interpretation of the final part of the classical proof -- Theorem \ref{thm-ramsey}. Here we show that $\RAMPAIR{c}$ follows directly from $\IPP$ given that we have already constructed a min-monochromatic set. Suppose we have interpreted the lemma $\IPP$, in other words: for any $\varepsilon\colon \BB\times\NN^\NN\to \NN$ and $c$ we can construct $x$ and $p$ satisfying
\begin{equation}\label{ndipp}\forall n\leq\varepsilon_xp(pn\geq n\wedge c(pn)=x).\end{equation} 
Assuming that we have already (ineffectively) produced the min-monochromatic set given by $a$, if $c^a$ is defined as in Theorem \ref{thm-ramsey} and we set $\varepsilon^a_xp=\eta_x(a\circ p)$ (where we recall that $\eta$ is a counterexample function for the finitary Ramsey's theorem as in (\ref{finramsey})), then by (\ref{ndipp}) there exist $x^a$ and $p^a$ satisfying
\begin{equation*}
\forall n \leq \eta_{x^a}(a\circ p^a)(p^a n \geq n \wedge c^a(p^a n) = x^a).
\end{equation*}
It is easy to see that our main theorem follows since setting $F=a\circ p^a$, for $k\leq\eta_xF$ we have
\[ Fk = a(p^a k)\geq p^a k \geq k \]
and, given $F i < F j$
\[ c(Fi,Fj)=c(a(p^a i),a(p^a j))=c^a(p^a i) = x^a. \] 

So what about $a$? The key observation is that we do not really need to have constructed the whole of $a$ for the above argument to work, only a finite approximation of $a$ is necessary. By inspection, provided that $a$ is min-monochromatic up to
\eqleft{\varphi a = \max_{i\leq\eta_{x^a}(a\circ p^a)} p^a(i)}
the claim above still holds. Therefore, if in addition we have interpreted the lemma $\ER{c}$, running it on the counterexample function $\varphi$ gives us a sufficiently large approximation of the min-monochromatic branch needed for an approximation of Ramsey's theorem on $\eta$. Denoting the quantifier-free matrix of the dialectica interpretation of $A \equiv \exists x \forall y A_D(x ; y)$ as $\nd{A}{x}{y}$, we illustrate this construction, very informally, with the inference
\begin{prooftree}
\AxiomC{$\lambda a \ . \ \nd{\IPP[a]}{p^a,x^a}{\varepsilon^a}$}
\AxiomC{$\nd{\ER{c}}{a}{\varphi}$}
\BinaryInfC{$\nd{\RAMPAIR{c}}{a\circ p^a,x^a}{\eta}$}
\end{prooftree}
making clear that the realiser for $\IPP$ is computed relative to the parameter $a$. In practise this means that we run our program for $\ER{c}$ once, calling on the interpretation of $\IPP[a]$ each time we wish to check that a candidate $a$ is suitable.

An entirely analogous procedure is involved, in turn, for interpreting $\ER{c}$ itself. $\ER{c}$ follows from $\WKL$ assuming the existence of a monotone Skolem function $\beta$ making the tree $T$ decidable. Therefore we need to calibrate exactly how much of $\beta$ is required in order to successfully run the computational interpretation of $\WKL$. As we will see, this part is rather more involved. A rough map of our whole construction is given in Figure \ref{fig-interpreted}.

\begin{figure}[h]
\begin{center}
{\footnotesize
\begin{prooftree}
\AxiomC{L. \ref{lemma-ipp}}
\noLine
\UnaryInfC{$\lambda a \ . \ \nd{\IPP[a]}{p^a,x^a}{\varepsilon^a}$}
\AxiomC{Th. \ref{wkl-main}}
\noLine
\UnaryInfC{$\lambda \beta \ . \ \nd{\WKL[\beta]}{\alpha^\beta}{\omega_\beta}$}
\AxiomC{L. \ref{ca-approx}}
\noLine
\UnaryInfC{$\nd{\pCA}{\beta}{\tilde q,\tilde\omega}$}
\RightLabel{Th. \ref{swkl-main}, L. \ref{swkl-lemma2}}
\BinaryInfC{$\nd{\ER{c}}{a^{\alpha,\beta}}{\varphi}$}
\RightLabel{Th. \ref{ramsey-main}}
\BinaryInfC{$\nd{\RAMPAIR{c}}{a\circ p^a,x^a}{\eta}$}
\end{prooftree}
}
\end{center}
\caption{Interpreted proof of Ramsey's theorem}
\label{fig-interpreted}
\end{figure}

By comparison with our proof tree in Section \ref{sec-formal} it is clear -- as expected -- that the structure of the interpreted proof reflects that of the classical proof. 

As mentioned in Remark \ref{remark-ncolours}, generalising our construction to the $n$-colour case becomes non-trivial in the construction of the min-monochromatic branch, as the E/R tree is no longer binary branching for $n>2$ and therefore calibrating how much of $\beta$ we require is a little more intricate. Also, in the $n$-colour case full use of $\IPP$ would be made. That is explained in Lemma \ref{lemma-ipp} below.

We now proceed with our formal program extraction. We interpret each of the main ineffective lemmas $\IPP$, $\WKL$ and $\pCA$ in turn using the product of selection functions, and combine these realisers as described above in order to produce an approximation of Ramsey's theorem. In Section \ref{subsec-gameint} we discuss the aforementioned link with sequential games, and give our program a game-theoretic reading.

\subsection{Interpreting $\WKL$}
\label{subsec-wkl}

The first ineffective step in the proof we examine is the use of weak K\"onig's lemma to produce the infinite sequence $\alpha$ given a Skolem function $\beta$, as in Lemma \ref{wkl-used}. We will show how to witness the no-counterexample interpretation of this lemma. As before, let $T$ be the $\Sigma^0_1$-predicate on $\BB^*$ defined as
\eqleft{T(s) := \exists k (\underbrace{\exists k' \!\in\! [|s|, k] \, \forall i < |s| (s_i = 0 \;\leftrightarrow\; i \prec k')}_{\mt(s, k)}).}
Let us assume we have an ideal Skolem function $\beta$ satisfying 
\begin{equation} \label{skolem-ideal}
\forall n, k \forall s (|s| = n \wedge \mt(s, k) \to \mt(s, \beta n)).
\end{equation}
Because the existence of $\beta$ is ineffective, we will keep track of exactly when we call on $\beta$ by highlighting it with a $\fbox{box}$. This means that we know how much of $\beta$ is needed to construct an approximation of $\alpha$, so that later we can in turn produce an approximation to $\beta$ sufficient for the construction of $\alpha$. 

Recall that we use the abbreviation $T^\beta(s) = \mt(s, \beta(|s|))$. The n.c.i. of Lemma \ref{wkl-used} is as follows 
\begin{equation} \label{wkl-inter-eq}
\forall\omega^{\BB^\NN\to \NN}\exists\alpha T^\beta(\initSeg{\alpha}{\omega\alpha}).
\end{equation}
Therefore, let us show how to witness $\alpha$ as a function of $\beta$ and $\omega$.

\begin{lemma} \label{wkl-premis} Let $\beta$ be a function satisfying (\ref{skolem-ideal}). The tree $T^\beta$ has branches of arbitrary length, i.e. for all $n$ there exists $s$ such that $|s| = n$ and
\[ \underbrace{\exists k' \!\in\! [n, \beta n] \, \forall i < |s| (s_i = 0 \;\leftrightarrow\; i \prec k')}_{T^\beta(s)}, \]
\end{lemma}
\begin{proof} Given $n$ define $s$ as the sequence of length $n$ such that, for $i < n$, $s_i = 0$ if and only if $i \prec n$. We then have $\mt(s, n)$. By (\ref{skolem-ideal}) with \fbox{$k = n$}, we can conclude $\mt(s, \beta n)$.
\end{proof}

\begin{lemma}\label{wkl-lemma-weak} Let $\INF_n(T) \equiv \exists s (|s| = n \wedge T(s))$. Let also $\beta$ be a function satisfying (\ref{skolem-ideal}), and $\varepsilon \colon J_{\BB} \BB$ be defined as
\eqleft{\varepsilon_s p \stackrel{\BB}{=}
\left\{
\begin{array}{ll}
	\True & {\rm if} \; \INF_{p(\True) + 1}(T^\beta_s) \to \INF_{p (\True)}(T^\beta_{s* \True}) \\[2mm]
	\False & {\rm otherwise}.
\end{array}
\right.
}
Then
\begin{equation}\label{sel-wkl} \forall s, p (\underbrace{\INF_{p(\varepsilon_s p) + 1}(T^\beta_s)}_{(i)} \to \INF_{p (\varepsilon_s p)}(T^\beta_{s* \varepsilon_s p})). \end{equation}
\end{lemma}

\begin{proof}Fix $s$ and $p$ and assume ($i$). If
\eqleft{\INF_{p(\True) + 1}(T^\beta_s) \to \INF_{p (\True)}(T^\beta_{s* \True})}
holds, then $\varepsilon_s p = \True$ and we are done. If, on the other hand, 
\eqleft{\underbrace{\INF_{p(\True) + 1}(T^\beta_s)}_{(ii)} \wedge \underbrace{\neg \INF_{p (\True)}(T^\beta_{s* \True})}_{(iii)}}
holds, then $\varepsilon_s p = \False$. Hence, the assumption ($i$) implies ($iv$) $\INF_{p(\False) + 1}(T^\beta_s)$. Now we consider two cases: \\[1mm]
\textbf{Case 1: $p(0)\geq p(1)$.} By $(ii)$ and $(iii)$ we have $\INF_{p (\True)}(T^\beta_{s\ast\False})$. Therefore by $(\mona)$ we have $\exists t^{\BB^{p(\False)}} T'(s\ast\False\ast t,\beta(|s|+p(\True)+1))$, and applying (\ref{skolem-ideal}) for \fbox{$n=|s|+p(\False)+1$ and $k=\beta(|s|+p(\True)+1)$} we obtain
\[ \exists t^{\BB^{p(\False)}} T'(s\ast\False\ast t,\beta(|s|+p(\False)+1)) \equiv \INF_{p(\False)}(T^\beta_{s\ast\False}). \]
\textbf{Case 2: $p(0)<p(1)$.} Applying (\ref{skolem-ideal}) on \fbox{$n=|s|+p(\True)+1$ and $k=\beta(|s|+p(\False)+1)$} and ($iii$) we obtain
\[ \forall t^{\BB^{p(\True)}}\neg T'(s\ast \True\ast t,\beta(|s|+p(\False)+1)). \]
By ($\mona$) we have $$\forall r^{\BB^{p(\False)}}\neg T'(s\ast \True\ast r,\beta(|s|+p(\False)+1))\equiv\neg\INF_{p(\False)}(T^\beta_{s\ast 0}).$$ But then by ($iv$) we obtain $\INF_{p(\False)}(T^\beta_{s\ast\False})$ and we are done. \end{proof}

\begin{remark} \label{sel-wkl-skolem} By inspecting the above proof we see that to verify that the selection functions $\varepsilon$ satisfy (\ref{sel-wkl}) for given $s$, $p$ it is sufficient that the Skolem function $\beta$ satisfies (\ref{skolem-ideal}) only up to $$n=|s|+\max\{p(\True),p(\False)\}+1 \ \mbox{and} \ k={\max}_{i\leq n} \beta(i).$$
\end{remark}

In order to construct a witness for (\ref{wkl-inter-eq}) we shall first build a sequence $\sbool$ satisfying
\begin{equation} \label{wkl-goal}
\forall k < \omega \sbool (\INF_{\omega \sbool - k}(T^\beta_{\initSeg{\sbool}{k}}) \to \INF_{\omega \sbool - k - 1}(T^\beta_{\initSeg{\sbool}{k+1}})).
\end{equation}
We will then obtain (\ref{wkl-inter-eq}) by a simple induction on $k$. 

\begin{theorem} \label{wkl-main} Let $\beta$ be a function satisfying (\ref{skolem-ideal}), and $\omega \colon \BB^\NN \to \NN$ be given. Define $q^\omega\sbool$ as $\omega\sbool-k-1$ where $k<\omega\sbool$ is the least refuting (\ref{wkl-goal}), and $0$ if no such $k$ exists. Also, let $\varepsilon$ be as defined in Lemma \ref{wkl-lemma-weak}. The sequence
\eqleft{\alpha = \EPS{\pair{\,}}{\omega}{\varepsilon}(q^\omega)}
satisfies $T^\beta(\initSeg{\sbool}{\omega \sbool})$.
\end{theorem}
\begin{proof} By Lemma \ref{wkl-lemma-weak} we have
\begin{equation} \label{wkl-step1}
\forall s, p (\INF_{p(\varepsilon_s p) + 1}(T^\beta_s) \to \INF_{p (\varepsilon_s p)}(T^\beta_{s* \varepsilon_s p})).
\end{equation}
By Theorem \ref{eps-main} we have that $\alpha$ (as above) and $p_{\initSeg{\sbool}{n}}$ (as defined in Theorem \ref{eps-main}) are such that, for $n \leq \omega \sbool$,
\eqleft{
\begin{array}{lcl}
\sbool n & = & \varepsilon_{\initSeg{\sbool}{n}} p_{\initSeg{\sbool}{n}} \\[2mm]
q^\omega \sbool & = & p_{\initSeg{\sbool}{n}}(\varepsilon_{\initSeg{\sbool}{n}} p_{\initSeg{\sbool}{n}}).
\end{array}
}
Hence, taking \fbox{$s = \initSeg{\sbool}{\omega\sbool-q^\omega\sbool-1}$ and $p = p_s$} in (\ref{wkl-step1}), we obtain
\begin{equation} \label{wkl-step2}
\INF_{q^\omega \sbool +1}(T^\beta_{\initSeg{\sbool}{\omega\sbool-q^\omega\sbool-1}}) \to \INF_{q^\omega \sbool}(T^\beta_{\initSeg{\sbool}{\omega\sbool-q^\omega\sbool}}).
\end{equation}
Therefore, by the definition of $q^\omega$ we must have that (\ref{wkl-goal}) holds. If not, then there is some least $k < \omega \alpha$ refuting (\ref{wkl-goal}), but then (\ref{wkl-step2}) is equivalent to 
\begin{equation*}
\INF_{\omega\sbool-k}(T^\beta_{\initSeg{\sbool}{k}}) \to \INF_{\omega\sbool-k-1}(T^\beta_{\initSeg{\sbool}{k+1}}).
\end{equation*}
Now, by Lemma \ref{wkl-premis} we have $\INF_{\omega \sbool}(T^\beta)$ (i.e. by taking \fbox{$n=\omega\alpha$ and $k=\omega\alpha$} in (\ref{wkl-step1}). Hence, by induction on $k$, from $k = 0$ to $k = \omega \alpha - 1$, we obtain $\INF_{0}(T^\beta_{\initSeg{\sbool}{\omega \sbool}})$, i.e. $T^\beta(\initSeg{\sbool}{\omega \sbool})$.
\end{proof}

Theorem \ref{wkl-main} defines a construction $\beta,\omega\mapsto\alpha^{\beta,\omega}$ that takes a Skolem function $\beta$ satisfying (\ref{skolem-ideal}) and a counterexample function $\omega$ and produces an ``approximately infinite" branch $\alpha$ of $T^\beta$. But the proof above only requires the selection functions $\varepsilon$ to satisfy (\ref{wkl-step1}) for the specific $s$, $p$ outlined, which in turn (Remark \ref{sel-wkl-skolem}) only require $\beta$ to satisfy (\ref{skolem-ideal}) for a finite number of inputs. Thus we obtain:

\begin{corollary}\label{skolem-bound}Given $\beta$ and $\omega$, let $\alpha$ and $p_s$ be constructed as in Theorem \ref{wkl-main} and define
\eqleft{
\begin{array}{lcl}
N^{\beta,\omega} & = & \max\{\omega\alpha,|\omega\alpha-q^\omega\alpha-1|+\max\{p_s(\True),p_s(\False)\}+1\} \\[2mm]
K^{\beta,\omega} & = & \max\{\omega\alpha,{\max}_{i\leq N^{\beta,\omega}} \beta(i)\}.
\end{array}
}
If $\beta$ is an \textit{approximate} Skolem function up to \fbox{$n=N^{\beta,\omega}$ and $k=K^{\beta,\omega}$} then $\alpha$ (from Theorem \ref{wkl-main}) satisfies $T^\beta(\initSeg{\sbool}{\omega \sbool})$. \end{corollary}

\subsection{Interpreting $\Pi^0_1$-countable choice}
\label{subsec-ac}

We have described a construction $\beta \mapsto \alpha$ which for each oracle for the Skolem function $\beta$ computes an approximation to the infinite binary branch $\alpha$. In Corollary \ref{skolem-bound} we argued that one only needs an approximation to $\beta$ in order for our construction to work. We now show how to compute such an approximation. We first need the following lemma: 

\begin{lemma}\label{swkl-sel} \label{wkl-lemma} Let $\delta_n \colon J_{\NN} \NN$ be defined as
\begin{equation} \label{ndswklprem2wit}
\delta_n p = p^i(0)
\end{equation}
where $i$ is the least $\leq 2^n$ such that, for all $s^{\BB^n}$, $\mt(s, p^{i+1}(0)) \to \mt(s, p^i(0))$. We have
\begin{equation} \label{ndswklprem2}
\forall\lth{s}{n}(\mt(s,p(\delta_n p)) \rightarrow \mt(s,\delta_n p))
\end{equation}
for arbitrary $n$, $p$.
\end{lemma}
\begin{proof} Note that (\ref{ndswklprem2}) holds by definition once we show that such $i \leq 2^n$ must exist. Assume, for the sake of a contradiction, that
\begin{itemize}
	\item[(I)] for all $i \leq 2^n$ there exists an $s^{\BB^n}$ such that $\mt(s, p^{i+1}(0))$ and $\neg \mt(s, p^i(0))$.
\end{itemize}
By monotonicity of $\mt$ on the second argument, (I) clearly implies that
\begin{itemize}
	\item[(II)] $p^i(0) < p(p^i(0))$, for all $0 \leq i \leq 2^n$.
\end{itemize}
Since, in (I), we have $2^n +1$ possible values for $i$ but only $2^n$ possible values for $s$, there must be an $s$ and distinct $i$ and $j$, say $i < i+1 \leq j$, such that $\mt(s, p^{i+1}(0))$ and $\neg \mt(s, p^j(0))$. By (II), however, that is a contradiction.
\end{proof}

We now show how to construct an arbitrary approximation to the Skolem function $\beta$. The next result can be viewed as the computational analogue of Lemma \ref{skolem-lemma}.

\begin{lemma}\label{ca-approx}Given arbitrary counterexample functionals $\tilde\omega,\tilde q\colon \NN^\NN\to\NN$, define
\eqleft{\beta=\EPS{\pair{}}{\tilde\omega}{\delta}(\tilde q)}
where $\delta$ is defined as in Lemma \ref{swkl-sel}. Then $\beta$ satisfies
\begin{equation} \label{skolem-approx}
\forall n\leq\tilde\omega\beta \;\forall s^{\BB^n}(\exists k\leq\tilde q\beta \; \mt(s,k)\to \mt(s,\beta n)).
\end{equation}
\end{lemma}

\begin{proof}By the main theorem on $\EPSs$ and (\ref{ndswklprem2}) we obtain $$\forall n\leq\tilde\omega\beta \; \forall s^{\BB^n}(\mt(s,\tilde q\beta)\to \mt(s,\beta n)).$$ By ($\monb$) we have $\forall k\leq\tilde q\beta(\mt(s,k)\to \mt(s,\tilde q\beta))$, therefore (\ref{skolem-approx}) follows.\end{proof}

Now that we can construct approximations to $\beta$ we are able to construct an approximation to an infinite branch of the $\Sigma^0_1$ tree $T$.

\begin{theorem} \label{swkl-main} For all $\omega \colon \BB^\NN \times \NN^\NN \to \NN$ there exists $\alpha$ and $\beta$ such that
\begin{equation} \label{swklram2}
\forall n \!\leq\! \omega \alpha \beta \underbrace{\exists k \!\in\! [n, \beta n] \, \forall i < n (\alpha(i) = 0 \;\leftrightarrow\; i \prec k)}_{T^\beta(\initSeg{\alpha}{n})}.
\end{equation}
\end{theorem}
\begin{proof} Let $\beta,\omega\mapsto\alpha^{\beta,\omega}$ denote the construction defined in Theorem \ref{wkl-main}. Define
\begin{equation*}\begin{array}{ccc}\tilde\omega\beta=N^{\beta,\omega} & \mbox{and} & 
\tilde q\beta = K^{\beta,\omega} \end{array}\end{equation*}
where $N^{\beta,\omega}$ and $K^{\beta,\omega}$ are defined as in Corollary \ref{skolem-bound}. Define $\beta=\EPS{\pair{}}{\tilde\omega}{\delta}(\tilde q)$ and $\alpha=\alpha^{\beta,\omega_\beta}$. We claim that these satisfy (\ref{swklram2}). 
By (\ref{skolem-approx}) and Corollary \ref{skolem-bound} we have $T^\beta(\initSeg{\alpha}{\omega\alpha\beta})$. Now suppose that $n\leq\omega\alpha\beta$. Then by (\ref{skolem-approx})
\begin{equation*}\mt(\initSeg{\alpha}{n},\beta(\omega\alpha\beta))\to \mt(\initSeg{\alpha}{n},\beta n)\equiv T^\beta(\initSeg{\alpha}{n}),\end{equation*}
and since by ($\mona$) $T^\beta(\initSeg{\alpha}{\omega\alpha\beta})\to \mt(\initSeg{\alpha}{n},\beta(\omega\alpha\beta))$ we are done.\end{proof}

We are now in a position where we can construct an arbitrarily long min-monochromatic sequence $a$, even if the length of the sequence is determined only after we have built $a$, as given by $\psi a$, as long as we are allowed to use $\psi$ in the construction of $a$.

\begin{lemma} \label{swkl-lemma2} For any $\psi$ there exists a function $a \colon \NN \to \NN$ such that for all $n \leq \psi a$
\begin{equation} \label{swklram3}
(n \leq a n) \wedge \forall i, j, k \!<\! n (ak < ai \wedge ak < aj \to c(a k, a i) = c(a k, a j)).
 \end{equation}
\end{lemma}
\begin{proof} First, define a parametrised $a^{\sbool, \swit}$ as in Lemma \ref{build-a}:
\begin{equation} \label{adefn}
\begin{array}{rl}
a^{\sbool, \swit} 0 &:= 0 \\[2mm]
a^{\sbool, \swit}(n+1) &:= \mu k \in [n, \beta(\beta n + 1)] \left( \sbool k = 0 \right).
\end{array}
\end{equation}
Then, take (cf. Remark \ref{build-a-remark}) 
\eqleft{\omega \alpha \beta = \max_{i\leq\psi(a^{\alpha, \beta})}(\max\{i,\beta i+1,\beta(\beta i+1)+1\})}
and let $\sbool$ and $\swit$ be as the Theorem \ref{swkl-main}. It is easy, following the same proof as in Lemma \ref{build-a}, to check that $a = a^{\sbool, \swit}$ satisfies (\ref{swklram3}).
\end{proof}

\begin{remark}\label{remark-ncolours2}For the $n$-colour case, the construction of $a$ is more complicated (cf. Remark \ref{remark-ncolours}) and $\omega$ will need to demand a larger approximation to $\beta$.\end{remark}

\subsection{Final arguments and $\IPP$}
\label{subsec-ipp}

Finally, the last non-constructive step in the proof is the use of the infinite pigeon-hole principle. Note that we in fact only make use of a particular instance of $\IPP$, namely $n = 2$. Nevertheless, we refer to the general $\IPP$ so it is easier to see how our construction can be generalised for arbitrarily many colours. 

\begin{lemma} \label{lemma-ipp} We have
\[ \forall \varepsilon^{\BB\times\NN^\NN\to\NN} \exists x^\BB, p^{\NN^\NN} \forall i \leq \varepsilon_x p (p i \geq i \wedge c(p i) = x). \]
\end{lemma}
\begin{proof} Given $\varepsilon_x$ define
\eqleft{\tilde \varepsilon_x p = \mu i \leq \varepsilon_x p \neg (p i \geq i \wedge c(p i) = x). }
Then let $(a_0, a_1) = (\tilde \varepsilon_0 \otimes \tilde \varepsilon_1)(\max)$ and $N = \max\{a_0, a_1\}$. By the main theorem on the product of selection functions we have $p_0$ and $p_1$ such that
\[ a_0 = \tilde \varepsilon_0 p_0 \quad \quad a_1 = \tilde \varepsilon_1 p_1 \quad \quad N = p_0 (a_0) = p_1(a_1). \]
Let $x = c(N)$ and $p = p_x$. Clearly, $p (\tilde \varepsilon_x p_x) = p a_x = p_x a_x = N \geq a_x$. Moreover, $c(p(\tilde \varepsilon_x p_x)) = c(p a_x) = c(N) = x$. Hence, by the definition of $\tilde \varepsilon_x$ we must have
\[ \forall i \leq \varepsilon_x p (p i \geq i \wedge c(p i) = x). \]
Note that essentially the same proof works for the $n$-colour case, where we have $n$ selection functions $\tilde\varepsilon_0,\ldots,\tilde\varepsilon_{n-1}$ accounting for each colour, and we take the finite product $(\tilde\varepsilon_0\otimes\ldots\otimes\tilde\varepsilon_{n-1})(\max)$.
\end{proof}

The theorem then follows by combining the construction of the min-monochromatic sequence with an application of $\IPP$.

\begin{theorem}\label{ramsey-main} Let a colouring $c \colon \NN^2 \to \BB$ be fixed. For any pair of selection functions $\eta_x \colon J_{\NN} \NN$ there exists $F \colon \NN \to \NN$ and $x \colon \BB$ (explicitly given in Section \ref{explicit}) such that
\[ \forall k \leq \eta_x F (k \leq F k \wedge \forall i, j \leq k (F i < F j \to c(Fi, Fj)=x)). \]
\end{theorem}
\begin{proof} Assume $c \colon [\NN]^2 \to \BB$ and $\eta_0 \colon J_{\NN} \NN$ and $\eta_1 \colon J_{\NN} \NN$ are given. For any function $a$ let $c^a(i) = c(a(i), a(i+1))$. 
%
%
Let $\varepsilon_x^a p = \eta_x(a \circ p)$, with $a \colon \NN \to \NN$ as a parameter. By Lemma \ref{lemma-ipp} we have that there exists $p^a$ and $x^a$ such that
\begin{equation} \label{pen2}
\forall i \!<\! \eta_{x^a}(a \circ p^a) (p^a(i) \geq i \wedge c^a(p^a(i)) = x^a).
\end{equation}
Let $\psi a = \max_{i\leq p^a(\eta_{x^a}(a \circ p^a))}p^a(i)$. By Lemma \ref{swkl-lemma2} there exists an $a \colon \NN \to \NN$ such that for all $n \leq p^a(\eta_{x^a}(a \circ p^a))$ we have $a n \geq n$ and 
\begin{equation} \label{pen1}
\forall i, j, k \!<\! n (ak < ai \wedge ak < aj \to c(a k, a i) = c(a k, a j)).
\end{equation}
Take $F = a \circ p^a$ and $x = x^a$. Therefore, for $k \leq \eta_x F = \eta_x(a \circ p^a)$ we have
\begin{itemize}
\item $p^a k \geq k$ by (\ref{pen2}) which, by the above implies that
	\[ F k = a(p^a k) \geq p^a k \geq k. \]
\item and, for $i, j \leq k$, given that $Fi < Fj$, we have
	\[ x \stackrel{(\ref{pen2})}{=} c^a(p^a(i)) = c(a(p^a(i)),a(p^a(i+1))). \]
	Hence
	\[ c(\underbrace{a(p^a i)}_{F i}, \underbrace{a(p^a j)}_{F j}) \stackrel{(\ref{pen1})}{=} c(a(p^a i), a(p^a(i) + 1)) \stackrel{(\ref{pen2})}{=} x. \]
\end{itemize}
\end{proof}

\subsection{Summarising the construction of $x$ and $F$}
\label{explicit}

From the proof of Theorem \ref{ramsey-main} we can read off the construction of $F$ and $x$ which we summarise here. Recall that the input to our problem is a colouring $c \colon \NN^2 \to \BB$ and a pair of selection functions $\eta_x \colon J_{\NN} \NN$. Also, recall the abbreviations
\eqleft{
\begin{array}{lcl}
\mt(s, k) & \equiv & \exists k' \!\in\! [|s|, k] \, \forall i < |s| (s_i = 0 \;\leftrightarrow\; i \prec k')) \\[2mm]
T^\beta(s) & \equiv & \mt(s, \beta(|s|)) \\[2mm]
\INF_n(T^\beta_s) & \equiv & \exists t (|t| = n \wedge T^\beta(s * t)).
\end{array}
}
(A) {\bf Construction of $x$ and $F$ given $a \colon \NN^\NN$}. First, assume a function $a \colon \NN^\NN$ given and let $c^a(i) = c(a(i), a(i+1))$ and $\varepsilon_x^a p = \eta_x(a \circ p)$. Define
\eqleft{\tilde \varepsilon_x p = \mu i \leq \varepsilon^a_x p \neg (p i \geq i \wedge c^a(p i) = x).}
Take $(k_0, k_1) = (\tilde \varepsilon_0 \otimes \tilde \varepsilon_1)(\max)$ and $x^a = c(\max\{k_0, k_1\})$ and
\eqleft{
p^a(k) =
\left\{
\begin{array}{ll}
\tilde \varepsilon_1 (\lambda k' . \max \{k, k'\}) & {\sf if} \; x^a = 0 \\[2mm]
\max \{k_0, k\} & {\sf if} \; x^a = 1.
\end{array}
\right.
}
and $F^a = a \circ p^a$. \\[2mm]
(B) {\bf Construction of $\alpha$ given $\beta \colon \NN^\NN$ and $\omega \colon \BB^\NN \times \NN^\NN \to \NN$}. Then, we construct a sequence $\alpha^{\beta, \omega} \colon \BB^\NN$ parametrised by $\beta \colon \NN^\NN$ and $\omega \colon \BB^\NN \times \NN^\NN \to \NN$ as follows. Let
\eqleft{q^{\beta, \omega} \sbool = \omega \sbool  \beta -k-1,}
where $k<\omega \sbool \beta$ is the least refuting 
\eqleft{\forall k < \omega \sbool \beta (\INF_{\omega \sbool \beta - k}(T^\beta_{\initSeg{\sbool}{k}}) \to \INF_{\omega \sbool \beta - k - 1}(T^\beta_{\initSeg{\sbool}{k+1}})),
}
and 
\eqleft{\varepsilon^\beta_s p \stackrel{\BB}{=}
\left\{
\begin{array}{ll}
	\True & {\rm if} \; \INF_{p(\True) + 1}(T^\beta_s) \to \INF_{p (\True)}(T^\beta_{s* \True}) \\[2mm]
	\False & {\rm otherwise}.
\end{array}
\right.
}
Define 
\eqleft{\alpha^{\beta, \omega} = \EPS{\pair{\,}}{\lambda \alpha. \omega \alpha \beta}{\varepsilon^\beta}(q^{\beta, \omega}).}
(C) {\bf Construction of $\beta$ given $\omega \colon \BB^\NN \times \NN^\NN \to \NN$ using (B)}. Using $\alpha^{\beta, \omega}$ we construct a sequence $\beta^{\omega} \colon \NN^\NN$ parametrised by $\omega \colon \BB^\NN \times \NN^\NN \to \NN$ only. Let $\delta_n \colon J_{\NN} \NN$ be 
\begin{equation*} 
\delta_n p = p^i(0)
\end{equation*}
where $i$ is the least $\leq 2^n$ such that, for all $s^{\BB^n}$, $\mt(s, p^{i+1}(0)) \to \mt(s, p^i(0))$, and
\eqleft{
\begin{array}{lcl}
\tilde\omega\beta & = & \max\{\omega \alpha^{\beta, \omega} \beta,|\omega \alpha^{\beta, \omega} \beta -q^{\beta, \omega} \alpha^{\beta, \omega}-1|+\max\{p(\True),p(\False)\}+1\} \\[2mm]
\tilde q\beta & = & \max\{\omega \alpha^{\beta, \omega} \beta, {\max}_{i\leq \tilde\omega\beta } \beta(i)\},
\end{array}
}
where
\eqleft{
\begin{array}{lcl}
p(x) & = & \selEmb{\EPS{s * x}{\lambda \alpha . \omega \alpha \beta}{\varepsilon^\beta}}(q^{\beta, \omega}_{s * x}) \\[2mm]
s & = & \initSeg{\alpha^{\beta, \omega}}{\omega \sbool^{\beta, \omega} \beta -q^{\beta, \omega} \sbool^{\beta, \omega} -1}.
\end{array}
}
Define
\eqleft{\beta^\omega = \EPS{\pair{}}{\tilde\omega}{\delta}(\tilde q).}
(D) {\bf Construction of $\omega$ using (A)}. We now construct the missing $\omega$ as
\eqleft{\omega \alpha \beta = \max_{i\leq\psi(a^{\alpha, \beta})}(\max\{i,\beta i+1,\beta(\beta i+1)+1\})}
where $\psi a = \max_{i\leq p^a(\eta_{x^a}(a \circ p^a))}p^a(i)$, with $p^a$ and $x^a$ as defined in (A), and
\eqleft{
a^{\sbool, \swit} n =
\left\{
\begin{array}{ll}
0 & {\sf if} \; n = 0 \\[2mm]
\mu k \in [n-1, \beta(\beta (n-1) + 1)] \left( \sbool k = 0 \right) & {\sf if} \; n > 0.
\end{array}
\right.
}
(E) {\bf Construction of $x$ and $F$ using (A) -- (D)}. Finally, take $\beta = \beta^\omega$ and $\alpha = \alpha^{\beta, \omega}$ and $a = a^{\alpha, \beta}$, so that $x$ and $F$ are defined as $x = x^a$ and $F = a \circ p^a$.

\section{A Game-Theoretic Reading of the Proof}
\label{subsec-gameint}

Following the discussion in Section \ref{sec-selection}, we know that each instance of $\EPSs$ used in our finitisation of Ramsey's theorem corresponds to the computation of an optimal strategy in a partially defined\footnote{We call a game $\game$ partially defined when not all three parameters $\varepsilon$, $q$ and $\omega$ are given, and write the open parameters in square brackets e.g. $\game[\varepsilon]$.} game. We now discuss the specific games corresponding to the main instances of $\EPSs$ used in our extracted program, and show how our constructive proof Ramsey's theorem can be understood in game-theoretic terms.

\subsubsection*{$\Pi^0_1$ countable choice: $\game^{\NN,\NN}_{\pCA}[\tilde q,\tilde\omega] = (\delta,\tilde q,\tilde \omega)$}

The game central to our interpretation is that corresponding to our use of countable choice. The selection functions $\delta_n$ defined in Lemma \ref{swkl-sel} implement a `no new branches' strategy, picking a number $i = \delta_n p$ satisfying $$\forall s^{\BB^n}(\mt(s,p(i))\to \mt(s,i))$$ i.e. there are no branches $s$ of $T$ which have a witness bounded by the outcome $p(i)$ which is not already bounded by the move $i$. 

For any outcome function $\tilde q$ and control function $\tilde\omega$, an optimal strategy in this case is a sequence $\beta$ satisfying, for all $n\leq\tilde\omega\beta$,
\[ \forall s^{\BB^n}(\mt(s,\tilde q\beta)\to \mt(s,\beta n)). \]
This means that every move $\beta n$ in the play $\beta$ (for $n\leq\tilde \omega \beta$) already bounds a witness for any branch $s$ of length $n$ which has a witness bounded by the final outcome $\tilde q \beta$. This optimal strategy is precisely the approximation to a monotone Skolem function we require.

\subsubsection*{Weak K\"{o}nig's lemma: $\game^{\BB,\NN}_{\WKL}[\omega] = (\varepsilon,q^\omega,\omega)$}

The interpretation of $\WKL$ applied to the decidable tree $T^\beta$ is interpreted by a binary game (where the set of possible moves at each round is $\BB$). The strategy $\varepsilon_s$ at position $s$ defined by the selection functions given in Lemma \ref{wkl-lemma-weak} is to pick a boolean $b$ such that if $s$ extends to a branch in $T$ of length $|s|+p(b)+1$ then $s\ast b$ also extends to a branch of length $|s|+p(b)+1$. 

Given $\omega$, by choosing $q^\omega$ suitably as in Theorem \ref{wkl-main}, the optimal strategy of $\game_\WKL$ determined by these selection functions is a sequence $\alpha$ such that for all $k\leq\omega\alpha$, whenever $\initSeg{\alpha}{k}$ extends to a branch of length $\omega\alpha$, so does $\initSeg{\alpha}{k+1}$. If $T^\beta$ is infinite then $\pair{}$ extends to a branch of length $\omega\alpha$. Hence, by induction the relevant part $\initSeg{\alpha}{\omega\alpha}$ of this optimal play must be in $T^\beta$, and is therefore an approximation to an infinite branch.

\subsubsection*{The infinite pigeonhole principle: $\game^{\NN,\NN}_{\IPP}[\varepsilon] = (\tilde\varepsilon,\max,2)$}

The game corresponding to $\IPP$ is a \textit{finite} game with two rounds (or $n$ rounds for the $n$-colour Ramsey's theorem). The strategy $\tilde\varepsilon$ at each round $x=0,1$ is to play the least move $i\leq\varepsilon_xp$ the outcome $p(i)$ of which satisfies
\[ p(i)<i\vee c(p(i)) \neq x. \]
We compute the optimal play $\pair{a_0,a_1}$, and its outcome is the maximum $N=\max\{a_0,a_1\}$. But then, at round $x=c(N)$ we have $$p_x(a_x)\geq a_x\wedge c(p_x(a_x))=x$$ since $p_x(a_x)=N$, which implies that the selection function $\tilde\varepsilon_x$ must fail to find a suitable candidate. But since we know that an optimal strategy must exist, the only explanation is that such a candidate does not exist, or in other words, $x,p_x$ form an approximation to the infinite pigeonhole principle. \\

Following the discussion at the beginning of the section, it is not too hard to visualise how these games combine to witness the functional interpretation of Ramsey's theorem. We compute an optimal strategy $\beta$ in the game $$\game_{\pCA}[\lambda\beta.K^{\beta,\omega},\lambda\beta.N^{\beta,\omega}]$$ where the outcome and control functions involve computing an optimal strategy $\alpha^\beta$ in the auxiliary game $$\game_{\WKL}[\omega_\beta]$$ on $T^\beta$. As a result we obtain two optimal strategies $\beta$, $\alpha^\beta$ that combine to form an approximation $a^{\alpha,\beta}$ to a min-monochromatic branch.

In addition, the control function $\omega_\beta$ is defined in terms of $\varphi a^{\alpha,\beta}$, which in turn involves computing an optimal strategy in a \textit{further} auxiliary game $$\game_{\IPP}[\lambda x,p.\eta_x(a^{\alpha,\beta}\circ p)]$$ where $\eta$ is our counterexample function for $\RAMPAIR{c}$, in order to produce $x^a$, $p^a$ required to compute $\varphi a$.

Therefore our program can be viewed in terms of the computational of optimal strategies in three symbiotic games: one central game corresponding to $\pCA$ and two nested auxiliary games that are run each time we call on the relevant counterexample functions.

The computation as a whole returns an optimal strategy $\beta$ of $\game_{\pCA}$ and an optimal strategy $\alpha^\beta$ of $\game_{\WKL}$ that combine to form a sequence $a^{\alpha,\beta}$, along with $p^a$, $x^a$ arising from optimal strategy in $\game_{\IPP}$. Our realiser for the functional interpretation of Ramsey's theorem $F=a^{\alpha,\beta}\circ p^{a^{\alpha,\beta}}$ and $x=x^{a^{\alpha,\beta}}$ can therefore be written in terms of optimal strategies in these three games. \\

\noindent\textbf{Acknowledgements.} The authors thank the two anonymous referees for their helpful comments and suggestions on the preliminary version of this paper. The first author gratefully acknowledges support of The Royal Society under grant 516002.K501/RH/kk, and the second author acknowledges the support of an EPSRC doctoral training grant.

\bibliographystyle{plain}

\bibliography{dblogic-ram}

\end{document}